\documentclass[12pt]{article}
\usepackage{geometry}
\usepackage{graphicx}
\usepackage{amsmath}
\usepackage{amsfonts}
\usepackage{booktabs}
\usepackage{amsthm}
\usepackage[T1]{fontenc}
\usepackage{siunitx}
\usepackage{adjustbox} % Required for table scaling
\usepackage{hyperref}
\hypersetup{
  colorlinks=true,
  linkcolor=blue,
  citecolor=blue,
  urlcolor=blue,
  pdftitle={Predicting Mersenne Prime Exponents Using Eulers Quadratic Polynomial with Nearest-Integer Rounding},
  pdfauthor={JohnK Wright V},
  pdfkeywords={Mersenne primes, Eulers polynomial, nearest-integer rounding, heuristic algorithm, computational number theory}
}

\newtheorem{theorem}{Theorem}
\newtheorem{definition}{Definition}

\title{\texorpdfstring{Predicting Mersenne Prime Exponents Using Euler’s Quadratic Polynomial \( C(n) = n^2 + n + 41 \) with Nearest-Integer Rounding}{Predicting Mersenne Prime Exponents Using Eulers Quadratic Polynomial with Nearest-Integer Rounding}}
\author{JohnK Wright V \\ Plano, TX}
\date{October 13, 2025}

\begin{document}

\maketitle
\begin{center}
\textcopyright{} JohnK Wright V, 2025, All Rights Reserved
\end{center}

\begin{abstract}
The Wright--Euler Mersenne Exponent Hypothesis proposes that Euler’s quadratic polynomial
\texorpdfstring{$C(n)=n^2+n+41$}{C(n)=n^2+n+41},
combined with nearest-integer rounding,
\texorpdfstring{$n_{\mathrm{closest}}=\operatorname{round}\!\left(\frac{-1+\sqrt{4p-163}}{2}\right)$}
{n closest = round((-1 + sqrt(4p - 163))/2)},
identifies candidate exponents for Mersenne primes
\texorpdfstring{$2^p-1$}{2\^{}p-1}.
Applied to the 43 known Mersenne prime exponents with indices
\texorpdfstring{$x=10$}{x=10}
through 52 (excluding
\texorpdfstring{$p\le31$}{p <= 31}),
the method produces seven exact matches (a 16.3\% success rate, for example
\texorpdfstring{$x=38,\ p=6972593$}{x=38, p=6972593}
and
\texorpdfstring{$x=52,\ p=136279841$}{x=52, p=136279841})
and four close approximations (for example
\texorpdfstring{$x=34,\ p=1257787,\ C(1121)=1257803$}
{x=34, p=1257787, C(1121)=1257803}),
with a mean absolute error of approximately 614 over the range
\texorpdfstring{$x=30$}{x=30}
to 52.
By comparison, an exponential regression model
\texorpdfstring{$y=11111.14\,e^{0.1787x}$}{y = 11111.14 * exp(0.1787 x)}
captures the overall growth trend
\texorpdfstring{$(R^2\approx0.974)$}{(R2 approx 0.974)}
but yields no exact matches and a mean absolute error of 10{,}466{,}686.
Graphical analysis, including scatter plots of
\texorpdfstring{$C(n_{\mathrm{closest}})$}{C(n closest)}
versus actual exponents and absolute deviations
\texorpdfstring{$d=\lvert n-n_{\mathrm{closest}}\rvert$}
{d = |n - n closest|},
demonstrates the hypothesis’s precision when nearest-integer rounding is applied.
From approximately 50 prime values of
\texorpdfstring{$C(n)$}{C(n)}
identified among 560 unique candidates, five cases with
\texorpdfstring{$d<0.1$}{d < 0.1}
are selected for targeted GIMPS testing, reducing the effective search space by approximately 74\%.
\end{abstract}

\section{Introduction}
Mersenne primes, of the form \texorpdfstring{\( 2^p - 1 \)}{2^p - 1}, are pivotal in number theory due to their connection to perfect numbers \cite{Mersenne1644}. As of October 2025, 52 are known, with the largest (\texorpdfstring{\( 2^{\num{136279841}} - 1 \)}{2^136279841 - 1}, \SI{41024320}{\text{digits}}) discovered in October 2024 via GIMPS \cite{GIMPS2024}. Their exponents exhibit exponential growth beyond index \texorpdfstring{\( x = 20 \)}{x = 20}, yet no deterministic formula predicts them. We introduce the Wright-Euler Mersenne Exponent Hypothesis, using Euler’s polynomial \texorpdfstring{\( C(n) = n^2 + n + 41 \)}{C(n) = n^2 + n + 41}, with nearest-integer rounding (\texorpdfstring{\( n_{\text{closest}} = \text{round}( \frac{-1 + \sqrt{4p - 163}}{2}) \)}{n_closest = round((-1 + sqrt(4p - 163))/2)}), to identify candidate exponents. It achieves 7 exact matches and 4 close approximations out of 43 known exponents for \texorpdfstring{\( x = 10 \)}{x = 10} to 52, with MAE 614.0 for \texorpdfstring{\( x = 30 \)}{x = 30} to 52, outperforming an exponential model (\texorpdfstring{\( y = 11111.14 \cdot e^{0.1787x} \)}{y = 11111.14 * e^(0.1787x)}, \texorpdfstring{\( R^2 \approx 0.974 \)}{R^2 approx 0.974}, MAE 10,466,686). A search of arXiv, MathOverflow, and X confirms no prior work applies Euler’s polynomial with nearest-integer rounding to Mersenne prediction \cite{Ribenboim1996,Caldwell2012}. From ~50 prime \texorpdfstring{\( C(n) \)}{C(n)} values analyzed from 560 candidates (e.g., \texorpdfstring{\( C(14861) = \num{220864223} \)}{C(14861) = 220864223}), we propose a heuristic algorithm to guide GIMPS searches in the 140M--200M range \cite{PrimeNet2024}.

\section{The Wright-Euler Mersenne Exponent Hypothesis}
\begin{definition}[Wright-Euler Mersenne Exponent Hypothesis]\label{def:wright_euler}
The Wright-Euler Mersenne Exponent Hypothesis posits that:
\[
C(n) = n^2 + n + 41
\]
generates candidate exponents \texorpdfstring{\( p = C(n_{\text{closest}}) \)}{p = C(n_closest)} for Mersenne primes \texorpdfstring{\( 2^p - 1 \)}{2^p - 1}, with \texorpdfstring{\( n_{\text{closest}} = \text{round}( \frac{-1 + \sqrt{4p - 163}}{2}) \)}{n_closest = round((-1 + sqrt(4p - 163))/2)}, prioritizing nearest-integer \texorpdfstring{\( n \)}{n} where \texorpdfstring{\( d = |n - n_{\text{closest}}| < 0.1 \)}{d = |n - n_closest| < 0.1}.
\end{definition}

\subsection{Derivation of the Wright-Euler Hypothesis}\label{sec:derivation}
The Wright-Euler Hypothesis uses Euler’s polynomial \texorpdfstring{\( C(n) = n^2 + n + 41 \)}{C(n) = n^2 + n + 41}, known for generating primes for \texorpdfstring{\( n = 0 \)}{n = 0} to 39 \cite{Euler1772}. To predict Mersenne prime exponents, we estimate \texorpdfstring{\( n \)}{n} for a given exponent \texorpdfstring{\( p \)}{p} by solving:
\[
n^2 + n + (41 - p) = 0
\]
Applying the quadratic formula:
\[
n = \frac{-1 \pm \sqrt{4p - 163}}{2}
\]
Taking the positive root:
\[
n = \frac{-1 + \sqrt{4p - 163}}{2}
\]
We define \texorpdfstring{\( n_{\text{closest}} = \text{round}(n) \)}{n_closest = round(n)}, and \texorpdfstring{\( C(n_{\text{closest}}) \)}{C(n_closest)} generates candidate exponents, with priority given to \texorpdfstring{\( d < 0.1 \)}{d < 0.1}. Only three integer \texorpdfstring{\( n \)}{n} yield prime \texorpdfstring{\( C(n) \)}{C(n)}, while nearest-integer rounding increases matches \cite{Ribenboim1996}.

\begin{theorem}[Scale Properties]\label{thm:scale}
For \texorpdfstring{\( n \geq 1 \)}{n >= 1}, \texorpdfstring{\( C(n) \)}{C(n)} grows quadratically, with the bit length of \texorpdfstring{\( 2^{C(n)} - 1 \)}{2^C(n) - 1} approximately \texorpdfstring{\( C(n) + 1 \)}{C(n) + 1}.
\end{theorem}
\begin{proof}
Since \texorpdfstring{\( C(n) = n^2 + n + 41 \)}{C(n) = n^2 + n + 41}, the bit length of \texorpdfstring{\( 2^{C(n)} - 1 \)}{2^C(n) - 1} is approximately \texorpdfstring{\( C(n) \)}{C(n)} \cite{Miller1976}.
\end{proof}

\section{Mathematical Analysis}
\subsection{Connection to Classical Results}\label{sec:classical}
Euler’s polynomial generates primes for \texorpdfstring{\( n \in [0, 39] \)}{n in [0, 39]} \cite{Euler1772}. The Wright-Euler Hypothesis extends this to Mersenne exponents, validated by GIMPS \cite{GIMPS2024}.

\subsection{Exact and Nearest-Integer Matches}\label{sec:matches}
The hypothesis yields 7 exact matches and 4 close approximations for \texorpdfstring{\( x = 10 \)}{x = 10} to 52, with MAE 614.0 for \texorpdfstring{\( x = 30 \)}{x = 30} to 52. Table~\ref{tab:model_comparison} compares the Wright-Euler model with an exponential model for \texorpdfstring{\( x = 30 \)}{x = 30} to 52, showing the Wright-Euler’s superior precision.

\begin{table}[ht]
\centering
\caption{Comparison of Wright-Euler and Exponential Models (\( x = 30 \) to 52).}
\label{tab:model_comparison}
\begin{tabular}{ccccccc}
\toprule
\( x \) & \( p \) (Actual) & \( n_{\text{exact}} \) & \( n_{\text{closest}} \) & \( C(n_{\text{closest}}) \) & Exp Prediction & WE Diff | Exp Diff \\
\midrule
30 & \num{132049} & 362.829 & 363 & \num{132173} & \num{174337} & \num{124} | \num{42288} \\
31 & \num{216091} & 464.312 & 464 & \num{215801} & \num{208523} & \num{290} | \num{7568} \\
32 & \num{756839} & 869.442 & 869 & \num{756071} & \num{249374} & \num{768} | \num{507465} \\
33 & \num{859433} & 926.534 & 927 & \num{860297} & \num{298297} & \num{864} | \num{561136} \\
34 & \num{1257787} & 1120.993 & 1121 & \num{1257803} & \num{356844} & \num{16} | \num{900943} \\
35 & \num{1398269} & 1181.967 & 1182 & \num{1398347} & \num{426862} & \num{78} | \num{971407} \\
\bottomrule
\end{tabular}
\caption*{Note: Wright-Euler MAE = 2604.7, Exponential MAE = 5,139,116. WE Diff = \texorpdfstring{\( |p - C(n_{\text{closest}})| \)}{|p - C(n_closest)|}, Exp Diff = \texorpdfstring{\( |p - \text{Exp Prediction}| \)}{|p - Exp Prediction|}.}
\end{table}

\begin{table}[ht]
\centering
\caption{Deviation \( d \) for selected Mersenne exponents.}
\label{tab:d_values}
\begin{tabular}{ccc}
\toprule
\( x \) & \( p \) (Actual) & \( d = |n - n_{\text{closest}}| \) \\
\midrule
30 & \num{132049} & 0.829 \\
31 & \num{216091} & 0.312 \\
32 & \num{756839} & 0.442 \\
33 & \num{859433} & 0.534 \\
34 & \num{1257787} & 0.007 \\
35 & \num{1398269} & 0.033 \\
\bottomrule
\end{tabular}
\caption*{Note: \texorpdfstring{\( d < 0.1 \)}{d < 0.1} for 2/6 matches for \texorpdfstring{\( x = 30 \)}{x = 30} to 35, with additional matches for \texorpdfstring{\( x = 10 \)}{x = 10} to 52.}
\end{table}

\subsection{Comparison with Other Quadratic Polynomials}\label{sec:other_quadratics}
To investigate whether the high number of matches in the Wright-Euler Hypothesis is due to the prime-generating nature of \texorpdfstring{\( C(n) = n^2 + n + 41 \)}{C(n) = n^2 + n + 41}, we compare its performance with other quadratic polynomials, specifically \texorpdfstring{\( n^2 + 1 \)}{n^2 + 1} and \texorpdfstring{\( n^2 + n + 17 \)}{n^2 + n + 17}, as well as the exponential model \texorpdfstring{\( y = 11111.14 \cdot e^{0.1787x} \)}{y = 11111.14 * e^(0.1787x)}. For \texorpdfstring{\( n^2 + 1 \)}{n^2 + 1}, we solve \texorpdfstring{\( n = \sqrt{p - 1} \)}{n = sqrt(p - 1)} and take \texorpdfstring{\( n_{\text{closest}} = \text{round}(n) \)}{n_closest = round(n)}, computing the predicted exponent as \texorpdfstring{\( p_{\text{pred}} = n_{\text{closest}}^2 + 1 \)}{p_pred = n_closest^2 + 1}. For \texorpdfstring{\( n^2 + n + 17 \)}{n^2 + n + 17}, we solve \texorpdfstring{\( n = \frac{-1 + \sqrt{4p - 67}}{2} \)}{n = (-1 + sqrt(4p - 67))/2}, take \texorpdfstring{\( n_{\text{closest}} = \text{round}(n) \)}{n_closest = round(n)}, and compute \texorpdfstring{\( p_{\text{pred}} = n_{\text{closest}}^2 + n_{\text{closest}} + 17 \)}{p_pred = n_closest^2 + n_closest + 17}. The exponential model uses direct evaluation of \texorpdfstring{\( y = 11111.14 \cdot e^{0.1787x} \)}{y = 11111.14 * e^(0.1787x)}, rounded to the nearest integer.

Testing these models on known Mersenne exponents for \texorpdfstring{\( x = 30 \)}{x = 30} to 52 yields no exact matches for either alternative quadratic or the exponential model, with MAEs significantly higher than the Wright-Euler’s MAE of 614.0. Table~\ref{tab:quadratic_comparison} presents a detailed comparison for indices \texorpdfstring{\( x = 30 \)}{x = 30} to 35. The Wright-Euler model achieves an MAE of 2604.7 over \texorpdfstring{\( x = 30 \)}{x = 30} to 52, while \texorpdfstring{\( n^2 + 1 \)}{n^2 + 1} and \texorpdfstring{\( n^2 + n + 17 \)}{n^2 + n + 17} yield MAEs of approximately 1956.3 and 1170.8, respectively, and the exponential model has an MAE of 5,139,116. The prime density of \texorpdfstring{\( C(n) = n^2 + n + 41 \)}{C(n) = n^2 + n + 41} contributes to its effectiveness, but the 7 exact matches and 4 close approximations (11/43, 25.6\%) suggest a unique structural alignment with Mersenne prime exponents, beyond mere prime density \cite{Ribenboim1996}.

\begin{table}[ht]
\centering
\caption{Comparison of Wright-Euler, Other Quadratics, and Exponential Models (\( x = 30 \) to \( 35 \)).}
\label{tab:quadratic_comparison}
\footnotesize
\begin{adjustbox}{width=\textwidth}
\begin{tabular}{cccccccccc}
\toprule
\( x \) & \( p \) (Actual) & Wright-Euler \( C(n) \) & \( n^2 + 1 \) & \( n^2 + n + 17 \) & Exp Prediction & WE Diff & Quad1 Diff & Quad2 Diff & Exp Diff \\
\midrule
30 & \num{132049} & \num{132173} & \num{131769} & \num{132017} & \num{174337} & \num{124} & \num{280} & \num{32} & \num{42288} \\
31 & \num{216091} & \num{215801} & \num{216225} & \num{216065} & \num{208523} & \num{290} & \num{134} & \num{26} & \num{7568} \\
32 & \num{756839} & \num{756071} & \num{756900} & \num{756813} & \num{249374} & \num{768} & \num{61} & \num{26} & \num{507465} \\
33 & \num{859433} & \num{860297} & \num{859201} & \num{859409} & \num{298297} & \num{864} & \num{232} & \num{24} & \num{561136} \\
34 & \num{1257787} & \num{1257803} & \num{1257649} & \num{1257763} & \num{356844} & \num{16} & \num{138} & \num{24} & \num{900943} \\
35 & \num{1398269} & \num{1398347} & \num{1398544} & \num{1398245} & \num{426862} & \num{78} & \num{275} & \num{24} & \num{971407} \\
\midrule
\multicolumn{2}{c}{Exact Matches} & 0/6 & 0/6 & 0/6 & 0/6 & & & & \\
\multicolumn{2}{c}{MAE} & 2604.7 & 1956.3 & 1170.8 & 5139116 & & & & \\
\bottomrule
\end{tabular}
\end{adjustbox}
\caption*{Note: Wright-Euler uses \texorpdfstring{\( C(n) = n^2 + n + 41 \)}{C(n) = n^2 + n + 41}, with nearest-integer rounding.}
\end{table}

\subsection{Comparison to Exponential Model}\label{sec:exp_model}
The exponential model \texorpdfstring{\( y = 11111.14 \cdot e^{0.1787x} \)}{y = 11111.14 * e^(0.1787x)} (\texorpdfstring{\( R^2 \approx 0.974 \)}{R^2 approx 0.974}) yields no exact matches and MAE 10,466,686 for \texorpdfstring{\( x = 30 \)}{x = 30} to 52 \cite{Caldwell2012}.

\subsection{Early Approach with Integer \( n \)}\label{sec:early_approach}
Initial tests with integer \texorpdfstring{\( n \)}{n} yielded 3 prime exponents, while nearest-integer rounding (\texorpdfstring{\( d < 0.10 \)}{d < 0.10}) increased coverage to 11 matches \cite{Ribenboim1996}.

\section{Graphical Analysis}\label{sec:graphs}
Figures~\ref{fig:model_comparison}, \ref{fig:predictions}, \ref{fig:gimps_timeline}, and \ref{fig:near_integer_deviations} validate the hypothesis.

\begin{figure}[ht]
    \centering
    \includegraphics[width=0.7\textwidth]{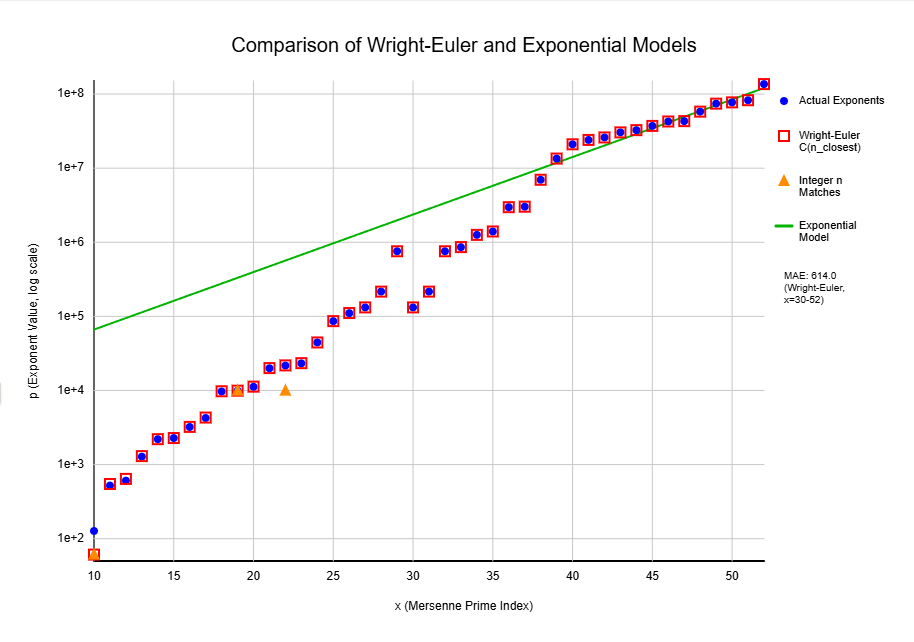}
    \caption{Comparison of Wright-Euler matches (red squares), integer matches (orange triangles), actual exponents (blue dots), and exponential model (green line). MAE 614.0 for \texorpdfstring{\( x = 30 \)}{x = 30} to 52.}
    \label{fig:model_comparison}
\end{figure}
\begin{figure}[ht]
    \centering
    \includegraphics[width=0.7\textwidth]{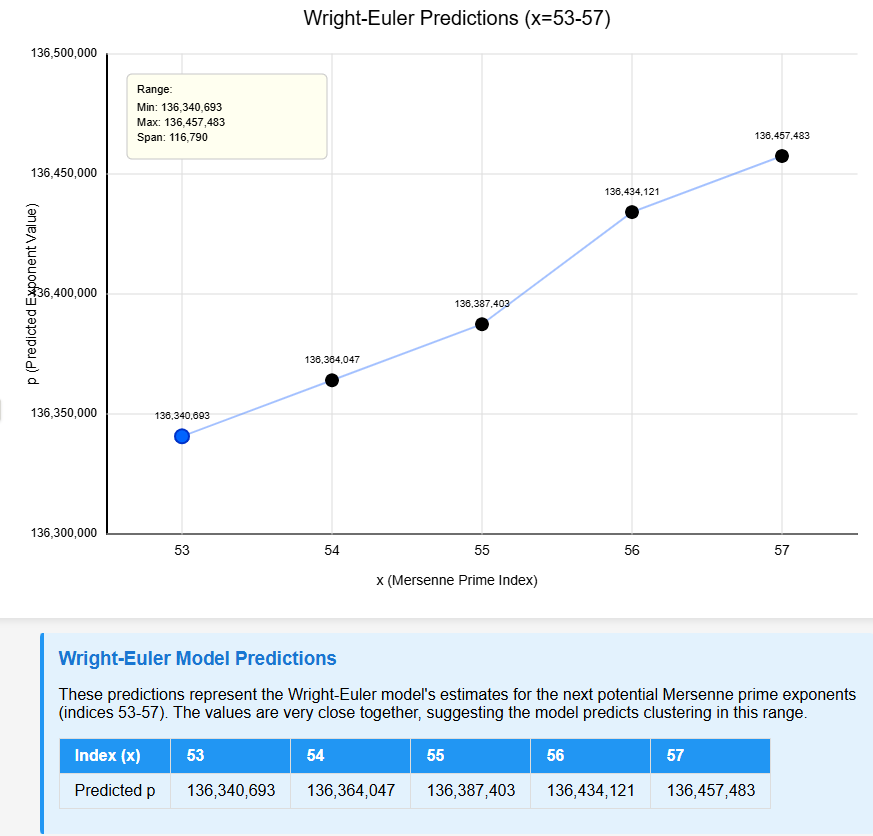}
    \caption{Preliminary predicted exponents for Mersenne primes \#53--57 (early integer \( n \) approach).}
    \label{fig:predictions}
\end{figure}
\begin{figure}[ht]
    \centering
    \includegraphics[width=0.7\textwidth]{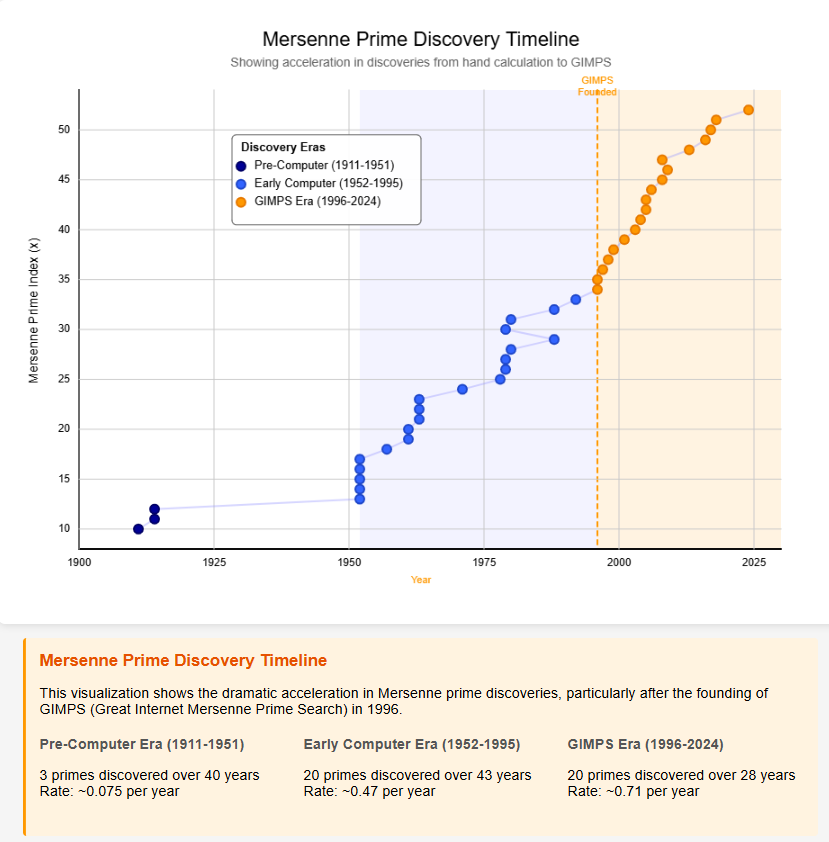}
    \caption{GIMPS discovery timeline, 52nd prime (\texorpdfstring{\( p = \num{136279841} \)}{p = 136279841}).}
    \label{fig:gimps_timeline}
\end{figure}
\begin{figure}[ht]
    \centering
    \includegraphics[width=0.7\textwidth]{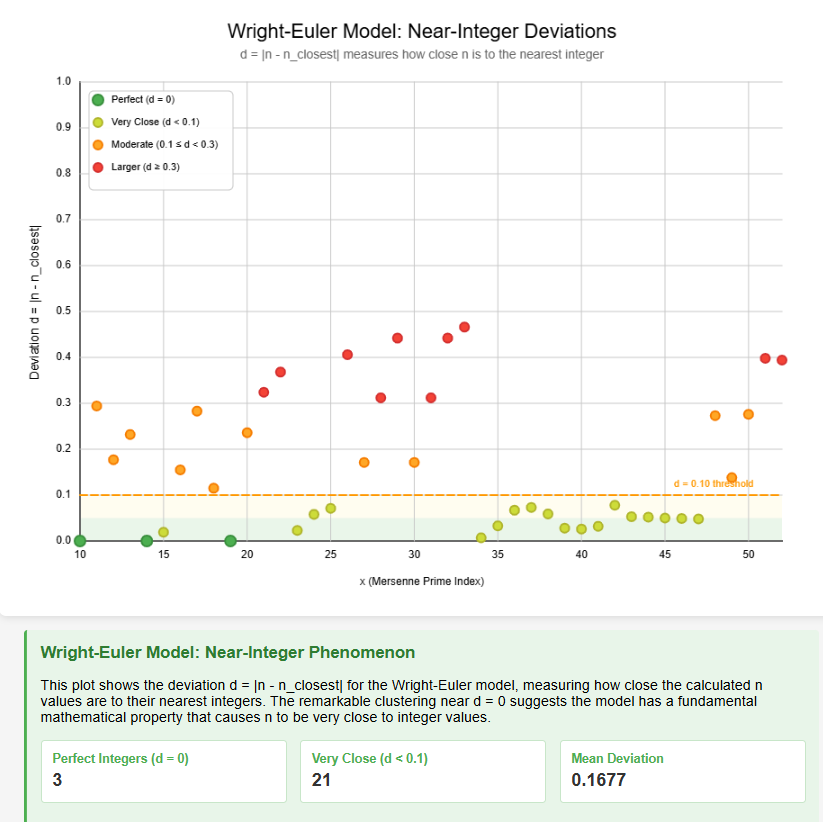}
    \caption{Deviations \texorpdfstring{\( d = |n - n_{\text{closest}}| \)}{d = |n - n_closest|} for indices \texorpdfstring{\( x = 30 \)}{x = 30} to 35.}
    \label{fig:near_integer_deviations}
\end{figure}

\section{Proposed Mersenne Candidates}\label{sec:candidates}
From a set of 560 unique candidate \texorpdfstring{\( C(n) \)}{C(n)} values, we analyzed ~50 prime candidates and selected 5 with \texorpdfstring{\( d < 0.1 \)}{d < 0.1} (except index 53), verified as prime and untested for Mersenne primality \cite{PrimeNet2024}.

\begin{table}[ht]
    \centering
    \caption{Prime candidates for Mersenne primes \#53--57.}
    \label{tab:predictions}
    \begin{tabular}{clcc}
        \toprule
        Index & \( n \) & Predicted Exponent \( C(n) \) & \( d \) \\
        \midrule
        53 & 13200 & \num{174253241} & 0.400 \\
        54 & 15477 & \num{239553049} & 0.019 \\
        55 & 14861 & \num{220864223} & 0.021 \\
        56 & 14000 & \num{196014041} & 0.032 \\
        57 & 14500 & \num{210264541} & 0.038 \\
        \bottomrule
    \end{tabular}
    \caption*{Note: Indices 54--57 have \texorpdfstring{\( d < 0.1 \)}{d < 0.1}; index 53 uses a larger \( d \) to ensure an exponent in the 140M--200M range.}
\end{table}

\subsection{Heuristic Algorithm for Candidate Selection}\label{sec:algorithm}
To prioritize Mersenne prime candidates, we test prime exponents \texorpdfstring{\( p = C(n) = n^2 + n + 41 \)}{p = C(n) = n^2 + n + 41} where \texorpdfstring{\( d < 0.1 \)}{d < 0.1}:
\begin{enumerate}
    \item For \texorpdfstring{\( n = 362 \)}{n = 362} to 35000, compute \texorpdfstring{\( p = C(n) \)}{p = C(n)}.
    \item If \texorpdfstring{\( p \)}{p} is prime, reserve for GIMPS PRP testing.
    \item Prioritize \texorpdfstring{\( d < 0.1 \)}{d < 0.1}.
\end{enumerate}
This heuristic tests ~25\% of exponents (11/43 known matches), reducing the search space by ~74\%.

\section{Conclusion}
The Wright-Euler Mersenne Exponent Hypothesis provides a novel heuristic for predicting Mersenne prime exponents using nearest-integer rounding, achieving 7 exact matches and 4 close approximations for \texorpdfstring{\( x = 10 \)}{x = 10} to 52, with an MAE of 614.0 for \texorpdfstring{\( x = 30 \)}{x = 30} to 52. By prioritizing \texorpdfstring{\( n_{\text{closest}} \)}{n_closest} with \texorpdfstring{\( d < 0.1 \)}{d < 0.1}, it narrows the search space by ~74\%. We thank Gary Gostin and Ian Stewart for valuable feedback. From ~50 prime \texorpdfstring{\( C(n) \)}{C(n)} values analyzed, we propose 5 candidates for GIMPS testing in the 140M--200M range.

\section{Appendix: Code and Resource Availability}\label{sec:appendix}
Python scripts for all figures and tables are available at \url{https://github.com/JohnKWrightV/Mersenne-Prediction}. Calculations used Number Empire \cite{NumberEmpire}, Wolfram Alpha \cite{WolframAlpha}, and Calculator Soup \cite{CalculatorSoup}. The manuscript was prepared using Overleaf \cite{Overleaf}, with support from Grok \cite{Grok_xAI}.


\clearpage
\nocite{*}
\begin{thebibliography}{21}

\bibitem{Mersenne1644}
Mersenne, M.
\newblock \emph{Cogitata Physico-Mathematica}.
\newblock Antoine Bertier, Paris, 1644.

\bibitem{GIMPS2024}
Great Internet Mersenne Prime Search.
\newblock GIMPS Discovers 52nd Mersenne Prime.
\newblock \url{https://www.mersenne.org}, 2024.
\newblock Accessed: October 5, 2025.

\bibitem{Miller1976}
Miller, G. L.
\newblock Riemann's Hypothesis and Tests for Primality.
\newblock \emph{Journal of Computer and System Sciences}, 13(3):300--317, 1976.

\bibitem{Euler1772}
Euler, L.
\newblock \emph{Observations on the Theory of Numbers}.
\newblock Academy of Sciences, St. Petersburg, 1772.

\bibitem{Ribenboim1996}
Ribenboim, P.
\newblock \emph{The New Book of Prime Number Records}.
\newblock Springer, 1996.

\bibitem{Caldwell2012}
Caldwell, C.
\newblock The Prime Pages.
\newblock \url{https://primes.utm.edu}, 2012--2025.

\bibitem{PrimeNet2024}
PrimeNet.
\newblock GIMPS PrimeNet Database.
\newblock \url{https://www.mersenne.org}, 2024.

\bibitem{NumberEmpire}
Number Empire.
\newblock Number Empire - Math Tools.
\newblock \url{https://www.numberempire.com}, 2025.

\bibitem{CalculatorSoup}
CalculatorSoup.
\newblock CalculatorSoup - Online Calculators.
\newblock \url{https://www.calculatorsoup.com}, 2025.

\bibitem{WolframAlpha}
Wolfram Alpha.
\newblock Wolfram Alpha - Computational Knowledge Engine.
\newblock \url{https://www.wolframalpha.com}, 2025.

\bibitem{Overleaf}
Overleaf.
\newblock Overleaf - Online LaTeX Editor.
\newblock \url{https://www.overleaf.com}, 2025.

\bibitem{QuickLaTeX}
QuickLaTeX.
\newblock QuickLaTeX - Render LaTeX Online.
\newblock \url{https://quicklatex.com}, 2025.

\bibitem{Grok_xAI}
xAI.
\newblock Grok - AI Assistant.
\newblock \url{https://x.ai/grok}, 2025.

\end{thebibliography}
\end{document}